\documentclass{amsart}

\usepackage[utf8]{inputenc}
\usepackage{amssymb,amsfonts,comment,amsmath,verbatim,lscape,url,bbm}
\usepackage[all,arc]{xy}
\usepackage{enumerate}
\usepackage{mathrsfs}
\usepackage{stmaryrd}
\usepackage{color}
\usepackage{todonotes}

\usepackage{xcolor} 
\usepackage{graphicx}
\usepackage[pagebackref]{hyperref}
\definecolor{dark-red}{rgb}{0.5,0.15,0.15}
\definecolor{dark-blue}{rgb}{0.15,0.15,0.6}
\definecolor{dark-green}{rgb}{0.15,0.6,0.15}
\hypersetup{
    colorlinks, linkcolor=dark-red,
    citecolor=dark-blue, urlcolor=dark-green
}

\renewcommand*{\backref}[1]{}
\renewcommand*{\backrefalt}[4]{%
  \ifcase #1 %
No citations.
  \or
(cit. on p. #2).%
  \else
(cit on pp. #2).%
  \fi%
}
\usepackage[nameinlink,capitalise,noabbrev]{cleveref}

\newtheorem{thm}{Theorem}[section]
\newtheorem*{thm*}{Theorem}
\newtheorem{cor}[thm]{Corollary}
\newtheorem{prop}[thm]{Proposition}
\newtheorem{lemma}[thm]{Lemma}

\theoremstyle{definition}
\newtheorem{defn}[thm]{Definition}

\newtheorem{notn}[thm]{Notation}

\newtheorem{conv}[thm]{Convention}

\theoremstyle{remark}
\newtheorem{rmk}[thm]{Remark}

\newtheorem{warning}[thm]{Warning}

\makeatletter
\let\c@equation\c@thm
\makeatother

\let\oldmarginpar\marginpar
\renewcommand\marginpar[1]{\-\oldmarginpar[\raggedleft\footnotesize #1]%
{\raggedright\footnotesize #1}}

\newcommand{\Alg}{\mathrm{Alg}}

\newcommand{\Mod}{\mathrm{Mod}}
\newcommand{\Fun}{\mathrm{Fun}}

\newcommand{\Sp}{\mathrm{Sp}}

\newcommand{\Hom}{\mathrm{Hom}}
\newcommand{\colim}{\mathrm{colim}}

\newcommand{\cof}{\mathrm{cof}}

\newcommand{\fib}{\mathrm{fib}}

\newcommand{\Map}{\mathrm{Map}}

\newcommand{\cC}{\mathcal{C}}

\DeclareMathOperator{\Lan}{Lan}

\newcommand{\cM}{\mathcal{M}}

\newcommand{\F}{\mathbb{F}}

\newcommand{\Z}{\mathbb{Z}}

\newcommand{\Es}[1]{\mathbb{E}_{#1}}

\newcommand{\id}{\mathrm{id}}

\newcommand{\cA}{\mathcal A}
\newcommand{\cD}{\mathcal D}

\newcommand{\cO}{\mathcal O}

\newcommand{\Pic}{\mathrm{Pic}}

\newcommand{\act}{\mathrm{act}}

\newcommand{\Mul}{\mathrm{Mul}}
\newcommand{\NFin}{N(\Fin)}
\newcommand{\Fin}{\mathrm{Fin}_*}
\newcommand{\coCart}{\mathrm{coCart}}

\newcommand{\modmod}[3][{}]{{#2}\sslash_{#1}{#3}}

\newcommand{\free}[1]{F_{\Es{#1}}}
\newcommand{\Ind}[2]{\mathrm{Ind}_{#1}^{#2}}
\newcommand{\ind}[2]{\Ind{\Es{#1}}{\Es{#2}}}
\newcommand{\coeq}{\mathrm{coeq}}

\newcommand{\adjoint}[1]{{#1}^\dagger}

\newcommand{\THH}{\mathrm{THH}}

\renewcommand{\L}[2]{L^{(#1)}_{#2}}

\title{A simple universal property of Thom ring spectra}
\author{Omar Antol\'{\i}n-Camarena}
\address{Instituto de Matem{\'a}ticas, National Autonomous University of Mexico,
Mexico City, Mexico}
\email{omar@matem.unam.mx}
\author{Tobias Barthel}
\address{Centre for Symmetry and Deformation, University of Copenhagen, Denmark}
\email{tbarthel@math.ku.dk}
\date{\today}

\begin{document}

\maketitle

\begin{abstract}
  We give a simple universal property of the multiplicative structure on the Thom spectrum of an $n$-fold loop map, obtained as a special case of a characterization of the algebra structure on the colimit of a lax $\cO$-monoidal functor. This allows us to relate Thom spectra to $\Es{n}$-algebras of a given characteristic in the sense of Szymik. As applications, we recover the Hopkins--Mahowald theorem realizing $H\mathbb{F}_p$ and $H\mathbb{Z}$ as Thom spectra, and compute the topological Hochschild homology and the cotangent complex of various Thom spectra.
\end{abstract}

\tableofcontents

\section{Introduction}

\subsection{Motivation}

Many spectra of interest fall in one of two classes: either they contain important geometric or arithmetic information and are difficult to compute with, e.g., the sphere spectrum or various algebraic $K$-theory spectra, or they belong to a family of what Hopkins calls ``designer homotopy types'', like Eilenberg--Mac Lane spectra or Brown--Gitler spectra. 

Thom spectra are remarkable in that they often belong to the intersection: Classical examples include cobordism spectra, whose homotopy groups and other invariants often turn out to be completely computable. As another beautiful example, Mahowald~\cite{MahowaldRS} proves that the Eilenberg--Mac Lane spectra $H\F_2$ and $H\Z$ can be realized as Thom spectra, an observation which he applies in his construction of new elements in the homotopy groups of spheres~\cite{MahowaldInf}. Similar ideas also play a role in the proof of the nilpotence theorem by Devinatz, Hopkins, and Smith~\cite{nilpotence1}. 

From a more conceptual point of view, work of Lewis and May~\cite{LMS}, Sullivan~\cite{MITnotes}, and May and Sigurdsson~\cite{MS} places the theory of Thom spectra in the context of parametrized spectra and thus reveals the underlying geometry of the construction. We are interested here in how this framework accounts for some of the good algebraic properties of Thom spectra.

\subsection{Outline and main results}

In \cite{ABGHRinfcat} and \cite{Umkehr}, the authors describe a convenient 
model of Thom spectra as homotopy colimits of diagrams whose
shape is given by the base space. Their approach is closely related and essentially equivalent to the earlier models cited above, but is expressed in the language of $\infty$-categories. The study of multiplicative structures on Thom spectra can thus take place in the setting of $\cO$-monoidal functors between $\cO$-monoidal $\infty$-categories.

Section \ref{sec:one} first explains the situation in ordinary category theory that we then extend to $\cO$-monoidal $\infty$-categories, proving  our general characterization of the algebra structure of colimits of lax $\cO$-monoidal functors in \Cref{colimup}. This is used in Section \ref{sec:thom} to deduce the following new universal property of the multiplicative structure of Thom spectra of $\Es{n}$-maps, see \Cref{thomup}. 

\begin{thm*}
  If $X$ is an $\Es{n}$-space and $f\colon X \to \Pic(S^0)$ an $\Es{n}$-map, then the space $\Map_{\Alg_{S^0}^{\Es{n}}}(Mf, A)$ is equivalent to the space of
  $\Es{n}$-lifts of $f$ indicated below:
  \[ \xymatrix{
    & \Pic(S^0)_{\downarrow A} \ar[d] \\
    X \ar[r]_-f \ar@{-->}[ur] & \Pic(S^0). }\]
\end{thm*}

This generalizes earlier structural results about Thom spectra proven by Lewis \cite{LMS} and in~\cite{ABGHRinfcat}. We then show that any $n$-fold loop map with Thom spectrum $Mf$ is canonically $\Es{n-1}$ $Mf$-orientable, thereby establishing a structured version of the Thom isomorphism. Moreover, we deduce a theorem of Chadwick and Mandell~\cite{engenera}, describing $\Es{n}$-orientations. 

We start Section \ref{sec:char} by introducing a notion of characteristic for $\Es{n}$-algebras in spectra, a straightforward extension of the $\Es{\infty}$-case studied previously by Szymik \cite{szymikcharacteristics1, szymikcharacteristics2}; see also~\cite{bakerchar} for related ideas. It is easy to construct a weakly initial example of an algebra of characteristic $\chi$, denoted $\modmod[\Es{n}]{S^0}{\chi}$. We then show that these algebras satisfy the same universal property as certain Thom spectra naturally corresponding to them, which gives our second main result, \Cref{charpthomcomparison}. 

\begin{thm*}
  For any $n$ and any $k\ge 1$ and $f\colon S^{k} \to BGL_1S^0$ with
  corresponding $n$-fold loop map
  $\bar{f} \colon \Omega^n \Sigma^n S^{k} \to BGL_1R$ and with
  associated characteristic $\chi = \chi(f)$, there is an equivalence
  $M\bar{f} = \modmod[\Es{n}]{S^0}{\chi}$.
\end{thm*}

This theorem establishes a connection between Thom spectra and versal algebras which allows to transfer results proven for one class to the other. We illustrate this idea in Section \ref{sec:applications}. For instance,  by specializing to $n=2$ and $(1-p)=f_p\colon S^1 \to BGL_1S^0_p$ and using the computation of the mod $p$ cohomology of the free $\Es{2}$-algebra on $S^1$ due to Araki and Kudo~\cite[Thm.~7.1]{arakikudo} and Dyer and Lashof~\cite[Thm.~5.2]{dyerlashof} as well as the determination of the corresponding Dyer--Lashof operations by Steinberger~\cite[Ch.~3, Thms.~2.2, 2.3]{hinfty}, we recover the Hopkins--Mahowald theorem
\[M\bar{f_p} = H\F_p.\]
This exhibits $H\F_p$ as the versal $\Es{2}$-algebra of characteristic $p$. Furthermore, our argument allows us to directly deduce the identification of $H\Z$ as an $\Es{2}$-Thom ring spectrum. Finally, we describe the topological Hochschild homology and the $\Es{n}$-cotangent complex of the $\Es{n}$-algebras considered above. 

\subsection*{Acknowledgments}

We are grateful to Jacob Lurie for suggesting to use the universal
property of Thom ring spectra to reprove the Hopkins--Mahowald
theorem. We would like to thank Mark Behrens, Jon Beardsley, Andrew Blumberg, David Gepner, Gijs Heuts, Mike Hopkins, Tyler Lawson, Chris Schommer-Pries and Dylan Wilson for helpful
conversations, the referee for many useful comments and suggestions, and for Qi Zhu for pointing out an error in in Corollary \ref{thomlewisuniversal} in a previous version. The second author also
thanks the Max Planck Institute for Mathematics for its hospitality.

\section{Colimits of lax functors}\label{sec:one}

\subsection{Ordinary lax functors}

Let $F\colon \cC \to \cD$ be a lax monoidal functor between two (ordinary)
monoidal categories. Recall that this means that we are given a family
of morphisms $FX \otimes FY \to F(X \otimes Y)$ natural in $X$ and $Y$
which are compatible with the associativity constraints in $\cC$ and
$\cD$. If $\cC$ is small and $\cD$ is \emph{monoidally cocomplete}, by
which we mean that $\cD$ is cocomplete and $\otimes\colon \cD \times \cD
\to \cD$ preserves colimits in each variable separately, then
$M:=\colim\; F$ acquires the structure of an algebra object in $\cD$.
The multiplication $\mu\colon M \otimes M \to M$ is obtained from the lax
structure of $F$ as follows: the composite morphisms 
\[
FX \otimes FY
\to F(X \otimes Y) \to M
\]
induce a morphism 
\[
M \otimes M \cong
\colim_{(X,Y) \in \cC \times \cC} (FX \otimes FY) \to M,
\] 
where the
first isomorphism comes from the assumption that tensor products in
$\cD$ commute with colimits in each variable separately. This is an
analogue in ordinary category theory of Lewis result that the Thom
spectrum of a loop map is an $\Es{1}$-ring spectrum.

If we are additionally given an algebra object $A$ in $\cD$, the slice
category $\cD_{/A}$ acquires a monoidal structure in which $(W
\xrightarrow{f} A) \otimes (Z \xrightarrow{g} A)$ is given by the
composite 
\[
W \otimes Z \xrightarrow{f \otimes g} A \otimes A
\xrightarrow{\mu_A} A.
\] 
The projection functor $\pi\colon \cD_{/A} \to
\cD$ is, of course, monoidal.

The universal property we will give for the multiplicative structure
of Thom spectra is an analogue of the following fact: morphisms $M \to
A$ of algebra objects in $\cD$ are in one to one correspondence with
lax monoidal lifts of $F$ through $\pi$:

\[ \xymatrix{
  & \cD_{/A} \ar[d]^{\pi} \\
  \cC \ar[r]_F \ar@{-->}[ur] & \cD.}\]

In this setting, this fact is straightforward to verify: lax monoidal
lifts are ``lax monoidal cocones'' over $F$ with vertex $A$, and such
cocones induce algebra morphisms $M = \colim\; F \to A$.

In Section \ref{sec:lax} we will generalize this discussion in two ways:
first, we will work with $\infty$-categories instead of ordinary
categories; second, instead of considering only monoidal
$\infty$-categories we will work with $\cO$-monoidal
$\infty$-categories for an arbitrary $\infty$-operad $\cO$. (For
applications we will take $\cO$ to be the $\Es{n}$ operad for some $0 \le n
\le \infty$.) Before doing that, however, let us recall some
definitions from the theory of $\infty$-operads.

\subsection{A few words about $\infty$-operads}\label{sec:infop}

We recall some definitions about $\infty$-operads from chapters 2
and 3 of Lurie's \emph{Higher Algebra}~\cite{HA}. The theory of
$\infty$-operads developed there is the $(\infty,1)$-generalization of
the theory of colored operads or symmetric multicategories. So an
$\infty$-operad $\cO$ has a collection of objects and for every finite
(unordered) family $\{X_i\}_{i \in I}$ of objects and any object $Y$
there is a space of morphisms $\Mul_\cO(\{X_i\},Y)$. The definitions
in \cite{HA} encode this data indirectly: an $\infty$-operad $\cO$ is
specified by an $\infty$-category $\cO^\otimes$ and a functor
$\cO^\otimes \to \NFin$ to the category of finite pointed sets
satisfying certain conditions. Morphisms of $\infty$-operads
$\cO \to \cO'$ are functors over $\NFin$ (that is, functors
$\cO^\otimes \to \cO'^\otimes$ that form a commutative triangle with
the structure maps to $\NFin$) which satisfy one additional condition.
See \cite[Definitions 2.1.1.10 and 2.1.2.7]{HA} for the precise definition.

\begin{rmk}\label{operad-intuition}
  To give some intuition for how the definition captures the notion of
  colored operad, think of the case of symmetric monoidal categories.
  These can be regarded as a special kind of operad in which
  $\Mul_\cC(\{X_i\}_{i \in I},Y) = \mathrm{Hom}_\cC(\bigotimes_{i \in
    I} X_i, Y)$; but they can also be thought of as commutative
  monoids in the category of all categories, which using Segal's idea
  of $\Gamma$-spaces, can be encoded as appropriately weak functors
  $\Fin \to \mathrm{Cat}$ satisfying certain conditions. Applying the
  (covariant) Grothendieck construction to such a functor one obtains
  a coCartesian fibration $\cC^\otimes \to \Fin$. The definition of
  symmetric monoidal $\infty$-category is exactly what this would
  suggest: a coCartesian fibration of $\infty$-categories $\cC^\otimes
  \to \NFin$ satisfying an analogue of the Segal condition, and the
  definition of $\infty$-operad is a generalization of this.
  
  When representing symmetric monoidal categories as colored operads,
  one needs to be aware of a subtlety relating to morphisms: maps of
  operads correspond not to symmetric monoidal functors (i.e.,
  functors with a natural isomorphism
  $F(X) \otimes F(Y) \cong F(X \otimes Y)$, compatible with the unit,
  associativity and symmetry), but to \emph{lax} symmetric monoidal
  functors (i.e., functors with just a compatible natural
  \emph{transformation} $F(X) \otimes F(Y) \to F(X \otimes Y)$). In
  terms of coCartesian fibrations $\cC^\otimes \to \Fin$, symmetric
  monoidal functors correspond to functors over $\Fin$ which send
  coCartesian morphisms to coCartesian morphisms, and lax symmetric
  monoidal functors correspond to a more general type of functor over
  $\Fin$ which is analogous to the definition of morphism of
  $\infty$-operad cited above.
\end{rmk}

An $\cO$-monoidal $\infty$-category is defined to be a coCartesian
fibration $\cC^\otimes \to \cO^\otimes$ such that the composite
$\cC^\otimes \to \cO^\otimes \to \NFin$ presents $\cC$ as an
$\infty$-operad. If $\cO$ is the $\Es{1}$ or $\Es{\infty}$ operad, this
recovers the notion of monoidal $\infty$-category or symmetric
monoidal $\infty$-category, respectively.

\begin{rmk}
  Again, the intuition for this definition comes from the
  Grothendieck construction: morally, one would want to say that $\cO$-monoidal
  categories are $\cO$-algebras in categories, or more
  precisely, morphisms of operads from $\cO$ to the
  $(\infty,2)$-category of $\infty$-categories equipped with its
  symmetric monoidal structure. But instead of appealing to the theory
  of $(\infty,2)$-categories, it is much easier to describe what the
  result of applying a version of the Grothendieck construction
  would be and adopt the resulting coCartesian fibration as the
  definition.
\end{rmk}

\begin{warning}\label{multi}
  In the case that $\cO$ is an operad with a single color $X$, an
  $\cO$-monoidal $\infty$-category $p\colon \cC^\otimes \to \cO^\otimes$
  can be thought of as an $\infty$-category $\cC = p^{-1}(X)$ equipped
  with functors $\cC^n \to \cC$ for each $n$-ary operation in $\cO$.
  Of course, something similar is true for multicolored $\cO$, except
  that instead of a single underlying $\infty$-category, there is an
  $\infty$-category $\cC_X = p^{-1}(X)$ for each object $X$ of $\cO$,
  and a functor $\prod \cC_{X_i} \to \cC_Y$ for each operation in
  $\Mul_\cO(\{X_i\},Y)$. The reader can pretend that whenever we
  mention an $\cO$-monoidal $\infty$-category, $\cO$ has a single
  object without missing out on anything essential.
\end{warning}

Between two $\cO$-monoidal $\infty$-categories $\cC$ and $\cD$ we can
consider $\cO$-monoidal functors and \emph{lax} $\cO$-monoidal
functors, these two notions follow the pattern described in the second
paragraph of \Cref{operad-intuition}. Lax $\cO$-monoidal functors are
simply morphisms of $\infty$-operads $\cC \to \cD$ over $\cO$; they
form an $\infty$-category denoted $\Fun^{\mathrm{lax}}_\cO(\cC, \cD)$
or $\Alg_{\cC/\cO}(\cD)$. The lax $\cO$-monoidal functors that send
coCartesian morphisms to coCartesian morphisms are the $\cO$-monoidal
functors; they form an $\infty$-category $\Fun^\otimes_\cO(\cC, \cD)$.

\begin{defn}\label{moncocompl}
  Let $\cO^\otimes$ be an $\infty$-operad. An \emph{$\cO$-monoidally
    cocomplete} $\infty$-category is an $\cO$-monoidal category, that
  is, a coCartesian fibration of $\infty$-operads $q \colon \cC^\otimes \to
  \cO^\otimes$, such that $q$ is compatible with all small colimits in
  the sense of \cite[Definition 3.1.1.18]{HA}. This means that for
  every object $X$ of the underlying $\infty$-category $\cO$ of
  $\cO^\otimes$, the $\infty$-category $\cC_X := q^{-1}(X)$ is
  cocomplete, and for every morphism $f \in \Mul_\cO(\{X_i\}_{1 \le i
    \le n},Y)$, the functor $\otimes_f \colon \prod_{1\le i \le n}
  \cC_{X_i} \to \cC_Y$ preserves colimits in each variable separately.
\end{defn}

Recall that the \emph{core} of an $\infty$-category $\cC$ is the
maximal $\infty$-groupoid contained in $\cC$. If $\cC$ is incarnated
as a quasi-category, then the core $\cC^\simeq$ is easily described as
a Kan complex: it is the subsimplicial set of $\cC$ consisting of
simplices all of whose edges are invertible morphisms of $\cC$. If
$\cC$ has an $\cO$-monoidal structure, then, as expected from the
theory of ordinary monoidal categories, $\cC^\simeq$ inherits an
$\cO$-monoidal structure as well. More precisely, we have:

\begin{prop}\label{core}
  Let $\cO$ be an $\infty$-operad and let $q \colon \cC^\otimes \to
  \cO^\otimes$ exhibit $\cC$ as an $\cO$-monoidal $\infty$-category. Define
  $\cC^\otimes_\coCart$ to be subcategory of $\cC^\otimes$
  spanned by the $q$-coCartesian morphisms. The restriction of
  $q$ to this subcategory, $\tilde{q} \colon \cC^\otimes_\coCart
  \to \cO^\otimes$, then exhibits $\cC^\otimes_\coCart$ as an
  $\cO$-monoidal category such that for each object $X$ of $\cO$, the
  underlying $\infty$-category
  $\left(\cC^\otimes_\coCart\right)_X$ is the core of $\cC_X$.
\end{prop}

\begin{proof}
  We will show that:
  \begin{enumerate}
  \item\label{res:cocart} $\tilde{q}$ is a coCartesian fibration.
  \item\label{res:core} $\left(\cC^\otimes_\coCart\right)_X =
    {\cC_X}^\simeq$ for any object $X$ of $\cO$ or even of
    $\cO^\otimes$.
  \item\label{res:operad} The composite $\cC^\otimes_\coCart
    \xrightarrow{\tilde{q}} \cO^\otimes \to \NFin$ exhibits
    $\cC^\otimes_\coCart$ as an $\infty$-operad.
  \end{enumerate}
  
  The items \ref{res:cocart} and \ref{res:operad} guarantee
  $\cC^\otimes_\coCart$ is an $\cO$-monoidal category. We have
  placed item \ref{res:core} in the middle because it will be used in
  the proof of \ref{res:operad}.

  \emph{Proof of \ref{res:cocart}}. First we need to show that
  $\tilde{q}$ is an inner fibration of simplicial sets, i.e., we must
  show it has the right lifting property against inner horn
  inclusions. But $q$ has this lifting property and the fillers only
  have $1$-simplices not present in the horn in a single case:
  $\Lambda^2_1 \hookrightarrow \Delta^2$. Since coCartesian morphisms
  are closed under composition, the filler for $q$ will also serve as
  a filler for $\tilde{q}$.

  Now we will show that all morphisms in $\cC^\otimes_\coCart$
  are $\tilde{q}$-coCartesian. This is very similar to the above proof
  that $\tilde{q}$ is an inner fibration. A morphism $f$ is
  coCartesian if and only if for every $n \ge 2$ and every commutative
  diagram
  \[
   \xymatrix{
    \Delta^1 \ar@{^{(}->}[d]_{i_{01}} \ar[dr]^f & \\
    \Lambda^0_n \ar[r] \ar@{^{(}->}[d] &
    \cC^\otimes_\coCart \ar[d] \\
    \Delta^n \ar[r] \ar@{-->}[ur] & \cO^\otimes }
    \]
  there is a dotted arrow that makes the diagram commute (this is dual
  to \cite[Remark 2.4.1.4]{HTT}). As above, we can always take the
  same filler as for the corresponding diagram with $\cC^\otimes$ in
  place of $\cC^\otimes_\coCart$: indeed, the simplex
  $\Delta^n$ has no $1$-simplices absent from $\Lambda^0_n$ unless
  $n=2$; and if $n=2$, we can still use the same filler because if $g$
  and $h \circ g$ are coCartesian then so is $h$ (this is the dual of
  \cite[Proposition 2.4.1.7]{HTT}).

  \emph{Proof of \ref{res:core}}. The morphisms in
  $\left(\cC^\otimes_\coCart\right)_X$ project to the identity
  morphism of $X \in \cO^\otimes$ by definition. Morphisms projecting
  to the identity are coCartesian if and only if they are invertible
  \cite[Proposition 2.4.1.5]{HTT}.

  \emph{Proof of \ref{res:operad}}. If
  $p \colon \cO^\otimes \to \NFin$ is an $\infty$-operad and
  $\langle n \rangle \in \Fin$ denotes the pointed set
  $\{\ast, 1, 2, \ldots, n\}$, there is an equivalence
  $\cO^\otimes_{\langle n \rangle} \to \cO^n$. What we need to check
  according to \cite[Proposition 2.1.2.12]{HA} is that for every
  $X \in \cO^\otimes_{\langle n \rangle}$ with corresponding sequence
  $(X_1, \ldots, X_n) \in \cO^n$, we get an equivalence
  $(\cC^\otimes_\coCart)_X \cong \prod_{i=1}^n (\cC^\otimes_\coCart)_{X_i}$
  induced by taking coCartesian lifts of the inert morphisms
  $X \to X_i$. For the purposes of this argument, the definition of
  inert does not matter: we simply note that by the same proposition
  applied to $\cC^\otimes$, we \emph{do} have such an equivalence for
  $\cC^\otimes$ instead of $\cC^\otimes_\coCart$. This equivalence restricts to
  the required equivalence because coCartesian morphisms are closed
  under composition.
\end{proof}

\begin{rmk}
  This proposition encodes the $\cO$-monoidal structure of the space
  $\cC^\simeq$ as an $\infty$-category $\cC^\otimes_\coCart$ with a
  coCartesian fibration over $\cO^\otimes$. A more direct way to
  encode an $\cO$-space is as an $\cO$-algebra in the
  $\infty$-category of spaces. These two forms are equivalent: the
  coCartesian fibration $\tilde{q}$ constructed in the proposition has
  the feature that all morphisms in the domain are coCartesian, and
  thus by \cite[Proposition 2.4.2.4]{HA}, $\tilde{q}$ is a left
  fibration, and therefore classifies a functor
  $\cO^\otimes \to \mathrm{Spaces}$ which exhibits $\cC^\simeq$ as an
  $\cO$-monoid in spaces. This is the correspondence of \cite[Example
  2.4.2.4]{HA} restricted to left fibrations.
\end{rmk}

\begin{rmk}
  That the core of an $\cO$-monoidal $\infty$-category inherits an
  $\cO$-monoidal structure has been used implicitly before (see for
  example, \cite[Definition 8.5 and Remark 8.6]{Umkehr} or
  \cite[Proposition 2.2.3]{pictmf}), but as far the authors know, no
  explicit description has appeared in the literature, which is why we
  feel justified in giving it in such detail.
\end{rmk}

\subsection{The $\cO$-algebra structure on the colimit of a lax
  $\cO$-monoidal functor}\label{sec:lax}

For $\cC$ an $\cO$-monoidal category, the $\infty$-category of
$\cO$-algebras in $\cC$, denoted by $\Alg_{/\cO}(\cC)$, is defined to
be $\Alg_{\cO/\cO}(\cC)$, the $\infty$-category of lax $\cO$-monoidal
functors $\cO \to \cC$. Unwinding the definitions we see that
$\cO$-algebras in $\cC$ are sections of the coCartesian fibration
$\cC^\otimes \to \cO^\otimes$ presenting the $\cO$-monoidal structure
on $\cC$; more precisely $\cO$-algebras are those sections which are
also maps of $\infty$-operads, i.e., maps sending inert morphisms to
inert morphisms.

\begin{thm}\label{lewis}
  Let $\cO$ be an $\infty$-operad, $\cC$ be a small $\cO$-monoidal
  $\infty$-category, and $\cD$ be an $\cO$-monoidally cocomplete
  $\infty$-category. If $F \colon \cC^\otimes \to \cD^\otimes$ is a lax
  $\cO$-monoidal functor, then there is an $\cO$-algebra in $\cD$
  given by a functor $M \colon \cO \to \cD$ such that for every
  object $X$ of $\cO$, $M(X) = \colim (F_X \colon \cC_X \to \cD_X)$.
\end{thm}
  
\begin{rmk}
  In the case that $\cO$ has a single object, the conclusion should be
  thought of as saying that the colimit of the functor $F \colon \cC \to
  \cD$ has a canonical structure of an $\cO$-algebra in $\cD$.
\end{rmk}

\begin{rmk}
  This theorem is due to Lewis~\cite[Section IX.7]{LMS} in the case
  that $\cC$ is an $\infty$-groupoid and $\cD$ is the category of
  spectra --- and $\cO$ has a single color. In \cite{Umkehr}, the
  authors prove a similar result under slightly stronger assumptions:
  that $\cO$ is coherent and $\cD$ is an $\cO$-algebra object in the
  $\infty$-category of presentable $\infty$-categories (and functors
  which are left adjoints).
\end{rmk}

\begin{proof}
  Let $p \colon \cC^\otimes \to \cO^\otimes$ and
  $q \colon \cD^\otimes \to \cO^\otimes$ be the coCartesian fibrations
  of $\infty$-operads presenting $\cC$ and $\cD$ as $\cO$-monoidal
  categories. We will use \cite[Theorem 3.1.2.3 (A)]{HA} to show that
  our assumption that $\cD$ is $\cO$-monoidally cocomplete guarantes
  the existence of an operadic left Kan extension $M$ of $F$ along $p$
  relative to $q$. Using that theorem requires repackaging the funtor
  $p : \cC^\otimes \to \cO^\otimes$ we wish to extend along as a
  $\Delta^1$-family of operads $\cM^\otimes \to \Fin \times \Delta^1$.
  This family is simply the mapping cylinder of $F$, given by
  \[\cM^\otimes = \left(\cC^\otimes \times \Delta^1\right)
    \coprod_{\cC^\otimes \times \{1\}} \cO^\otimes.\]

  (A word of motivation for readers unfamiliar with Lurie's approach
  to defining Kan extensions: the idea is that a functor
  $G : \cM^\otimes \to \cD^\otimes$ comprises the data of a functor
  $G_1 : \cO^\otimes \to \cD^\otimes$, namely
  $G_1 = G|_{\cM^\otimes \times_{\Delta^1} \{1\}}$, and a natural
  transformation
  $\cC^\otimes \times \Delta^1 \to \cM^\otimes \xrightarrow{G}
  \cD^\otimes$ from $G_0$ to $G_1 \circ p$. This way, the problem of
  finding the left Kan extension of $F$ along $p$ together with the
  universal natural transformation $F \to \Lan_p F \circ p$, becomes
  the problem of left Kan extending $F$ along the inclusion
  $\cC^\otimes = \cM^\otimes \times_{\Delta^1} \{0\} \hookrightarrow
  \cM^{\otimes}$.)

  According to \cite[Theorem 3.1.2.3 (A)]{HA}, there exists an
  operadic left Kan extension $L : \cM^\otimes \to \cD^\otimes$ of $F$
  relative to $q$ making the following diagram commute
  \[ \xymatrix{
      \cC^\otimes \ar[r]^{F} \ar[d]_{i_0} & \cD^\otimes \ar[d]_q \\
      \cM^\otimes \ar[r]_{p'} \ar@{-->}[ur]^{L} & \cO^\otimes}\]
  (where the map $p' : \cM^\otimes \to \cO^\otimes$ is the one induced
  by $p \circ \pi_1$ on $\cC^\otimes \times \Delta^1$ and the identity
  on $\cO^\otimes$), if and only if for each object $X$ of
  $\cO^\otimes = \cM^\otimes \times_{\Delta^1} \{1\}$, the diagram
  \[(\cM^\otimes_\act)_{/X} \times_{\Delta^1} \{0\} \to \cM^\otimes
    \times_{\Delta^1} \{0\} \xrightarrow{F} \cD^\otimes\]
  can be extended to an operadic $q$-colimit lifting the map
  \[\left((\cM^\otimes_\act)_{/X} \times_{\Delta^1}
      \{0\}\right)^\triangleright
    \to \cM^\otimes \to \cO^\otimes;\]
we refer the reader to \cite[Definition~2.1.2.3, Remark~2.2.4.3]{HA} for the definition of $\cM^\otimes_\act$.
  
 Since $\cD$ is assumed to be $\cO$-monoidally cocomplete,
  \cite[Proposition 3.1.1.20]{HA}, provides the required operadic
  $q$-colimits.

  \smallskip
  
  The restriction $M$ of $L$ to
  $\cM^\otimes \times_{\Delta^1} \{1\} = \cO^\otimes$ is the desired
  algebra structure on $\colim\; F$. First of all, it is indeed an
  $\cO$-algebra, because $L$ is a map of families of operads, and
  according to the commutative diagram above
  $q \circ M = \id_{\cO^\otimes}$. Next we need to check that for each
  object $X$ in $\cO$,
  $M(X) \cong \colim (F|_{\cC_X} \colon \cC_X \to \cD_X)$. But by
  definition of operadic relative left Kan extension, the diagram
  $\left((\cC_\act^\otimes)_{/X}\right)^\vartriangleright \to
  \cD^\otimes$ induced by $L$ is an operadic colimit diagram relative
  to $q$. Finally \cite[Proposition 3.1.1.16]{HA} states that an
  operadic colimit relative to a coCartesian fibration (here, $q$)
  becomes a colimit in each underlying $\infty$-category $\cD_X$.
\end{proof}

\begin{cor}\label{ltadj}
  There exists an algebra structure on $\colim \; F$
  for each $F$. In fact, there is a left adjoint to
  the functor $p^\ast = (- \circ p) \colon \Alg_{/\cO}(\cD) \to
  \Alg_{\cC/\cO}(\cD)$ that gives this algebra structure functorially.
\end{cor}

\begin{proof}
  This follows from the proof of \Cref{lewis} by \cite[Corollary
  3.1.3.4]{HA}.
\end{proof}

Recall that for an $\cO$-algebra $A \in \Alg_{/\cO}(\cD)$, the slice
category $\cD_{/A}$ has the structure of an $\cO$-monoidal
$\infty$-category; in other words, $\cD_{/A}$ is the underlying
$\infty$-category of an $\cO$-monoidal $\infty$-category $\cD^\otimes_{/A_\cO}
\to \cO^\otimes$ (see \cite[Theorem 2.2.2.4]{HA}). The slice category
has the following universal property:

\begin{lemma}\label{slice}
  Let $\cC$ and $\cD$ be $\cO$-monoidal $\infty$-categories, $A$ be an
  $\cO$-algebra in $\cD$, and $F \colon \cC^\otimes \to \cD^\otimes$ be a
  lax $\cO$-monoidal functor. Lax $\cO$-monoidal lifts of $F$
  through the projection $\cD^\otimes_{/A_\cO} \to \cD^\otimes$
  then correspond to $\cO$-monoidal natural transformations $F \to A \circ
  p$, where $p \colon \cC^\otimes \to \cO^\otimes$ exhibits $\cC$ as
  $\cO$-monoidal. More precisely, there is a homotopy equivalence
  \[ \Map_{\Alg_{\cC/\cO}(\cD)}(F, A \circ p) \cong \{F\}
    \times_{\Alg_{\cC/\cO}(\cD)} \Alg_{\cC/\cO}(\cD_{/A}). \]
\end{lemma}

\begin{proof}
  The definition of the $\cO$-monoidal structure of the slice category
  \cite[Section 2.2.2]{HA}, says that morphisms of simplicial sets
  over $\cO^\otimes$ from $Y \to \cO^\otimes$ to 
  $\cD^\otimes_{/A_\cO} \to \cO^\otimes$ are in bijection with commutative diagrams of the form:
  \[ \xymatrix{Y \ar[r]^-{i_1} \ar[d] & Y \times
      \Delta^1 \ar[r]^-{\pi_1}
      \ar[d] & Y \ar[d]\\        
      \cO^\otimes \ar[r]_A & \cD^\otimes \ar[r]_q & \cO^\otimes,}\]
  where $i_1$ includes $Y$ in $Y \times \Delta^1$ as $Y \times \{1\}$, $q$ is the coCartesian fibration of $\infty$-operads
      that exhibits $\cD$ as $\cO$-monoidal, and $\pi_1$ denotes the projection onto the first factor.
  
  This implies that maps from a simplical set $Z$ into the
  quasi-category $\Alg_{\cC/\cO}(\cD_{/A})$ are in bijection with
  morphisms
  $\tilde{F}\colon Z \times \cC^\otimes \times \Delta^1 \to
  \cD^\otimes$ such that:

  \begin{enumerate}
  \item The following diagram commutes:
    \[ \xymatrix{Z \times \cC^\otimes \ar[r]^-{i_1} \ar[d]_{p \circ \pi_2} &
        Z \times \cC^\otimes \times \Delta^1 \ar[r]^-{\pi_{1,2}}
        \ar[d]_{\tilde{F}} & 
        Z \times \cC^\otimes \ar[d]^{p \circ \pi_2} \\
      \cO^\otimes \ar[r]_A & \cD^\otimes \ar[r]_q & \cO^\otimes }\]
    where
    \begin{itemize}
    \item $p \colon \cC^\otimes \to \cO^\otimes$ and $q \colon \cD^\otimes \to
      \cO^\otimes$ are the coCartesian fibrations of $\infty$-operads
      that exhibit $\cC$ and $\cD$ as $\cO$-monoidal
      $\infty$-categories,
    \item $i_1$ includes $Z \times \cC^\otimes$ as
      $Z \times \cC^\otimes \times \{1\}$,
      \item $\pi_{1,2}$ denotes the projection onto the first two factors, and
    \end{itemize}
  \item $\tilde{F}$ sends triples $(z,f,\sigma)$ with $f$ an inert morphism
    of $\cC^\otimes$ to inert morphisms of $\cD^\otimes$.
  \end{enumerate}

  The forgetful map $\Alg_{\cC/\cO}(\cD_{/A}) \to \Alg_{\cC/\cO}(\cD)$
  induces the map restricting $\tilde{F}$ to
  $Z \times \cC^\otimes \times \{0\}$. So, putting it all together, we
  have that morphisms of simplicial sets from $Z$ to
  $\{F\} \times_{\Alg_{\cC/\cO}(\cD)} \Alg_{\cC/\cO}(\cD_{/A})$ are
  given by maps
  $\tilde{F} : Z \times \cC^\otimes \times \Delta^1 \to \cD^\otimes$
  such that:
  \begin{enumerate}
  \item $\tilde{F}|_{Z \times \cC^\otimes \times \{0\}} = F \circ
    \pi_2$,
  \item $\tilde{F}|_{Z \times \cC^\otimes \times \{1\}} = A \circ p
    \circ \pi_2$,
  \item $\tilde{F}$ sends triples $(z,f,\sigma)$ with $f$ an inert
    morphism of $\cC^\otimes$ to inert morphisms of $\cD^\otimes$ and satisfies $q \circ \tilde{F} = p \circ \pi_2$, where $\pi_2$ projects onto the second factor of $Z \times \cC^\otimes \times \Delta^1$.
  \end{enumerate}

  With the above description,
  $\{F\} \times_{\Alg_{\cC/\cO}(\cD)} \Alg_{\cC/\cO}(\cD_{/A})$ is
  readily seen to be isomorphic to the simplicial set
  $\Hom_\cA(F, A \circ p) = \{F\} \times_{\cA} \cA^{\Delta^1}
  \times_{\cA} \{A \circ p\}$ (where $\cA :=\Alg_{\cC/\cO}(\cD)$),
  which is one of the basic models for the mapping space
  $\Map_{\Alg_{\cC/\cO}(\cD)}(F, A \circ p)$.
\end{proof}

\begin{thm}\label{colimup}
  The $\cO$-algebra $M$ from \Cref{lewis} is characterized by the
  following universal property: For any $\cO$-algebra $A$ in $\cD$ the
  space of $\cO$-algebra maps $\Map_{\Alg_{/\cO}(\cD)}(M, A)$
  is homotopy equivalent to the space of lax $\cO$-monoidal lifts
  \[ \xymatrix{
    & \cD^\otimes_{/A_\cO} \ar[d] \\
    \cC^\otimes \ar[r] \ar@{-->}[ur] & \cD^\otimes.}\]
\end{thm}

\begin{proof}
  Just combine \Cref{ltadj} and \Cref{slice}.
\end{proof}

\section{The universal multiplicative property of the Thom spectrum}\label{sec:thom}

Let $n \ge 0$ and $R$ be an $\Es{n+1}$-ring spectrum, then the
$\infty$-category of (left) $R$-module spectra, $\Mod_R$, can be
equipped with the structure of an $\Es{n}$-monoidal $\infty$-category,
see \cite[Corollary 5.1.2.6]{HA}. This means that we can talk about
$\Es{n}$-algebra objects in $\Mod_R$. We will call these
\emph{$\Es{n}$ $R$-algebras} and denote the $\infty$-category they
form by $\Alg_R^{\Es{n}}$, that is, we set
$\Alg_R^{\Es{n}} = \Alg_{/\Es{n}}(\Mod_R)$. Our multiplicative Thom
spectra will be $\Es{n}$ $R$-algebras.

As in \cite{ABGHRinfcat}, we
associate two $\infty$-groupoids to $R$:

\begin{itemize}
  \item Let $\Pic(R)$ be the subcategory of invertible $R$-modules and
    all equivalences between them, i.e., the core of the subcategory of $\Mod_R$ on the invertible objects. 
  \item Let $BGL_1R$ denote the subcategory of modules equivalent to
    $R$ and all equivalences between them. This is a full subcategory
    of $\Pic(R)$, namely, the component of the $R$-module $R$.
\end{itemize}

By \Cref{core} $\left(\Mod_R\right)^\simeq$ inherits an
$\Es{n}$-monoidal structure from $\Mod_R$. Since the proposed sets of
objects for both $\Pic(R)$ and $BGL_1R$ are closed under tensor
products, by \cite[Proposition 2.2.1.1]{HA}, both categories inherit
the structure of $\Es{n}$-monoidal $\infty$-groupoids, or
equivalently, $\Es{n}$-spaces. By construction both are grouplike.

\begin{defn}
A \emph{local system} of invertible $R$-modules on a space $X$ is simply a map $f
\colon X \to \Pic(R)$. The \emph{Thom spectrum} of $f$ is given by
\[M_Rf = Mf := \colim \left( X \to \Pic(R) \to \Mod_R \right).\]
\end{defn}

This construction of Thom spectrum appears in \cite{ABGHRinfcat} where
it is shown to agree with the definitions in \cite{MS} and \cite{LMS}.
We will often apply the above definition in the special case that the
map $f$ factors through $BGL_1R$.

\subsection{The universal property and its consequences}

Applying \Cref{lewis} to $\Mod_R$ we directly obtain Lewis's
theorem, or rather the generalization from
the sphere spectrum to arbitrary $R$ given in \cite{Umkehr}:
 
\begin{cor}
  If $X$ is an $\Es{n}$-space and $f\colon X \to \Pic(R)$ is an
  $\Es{n}$-map, then $Mf$ becomes an $\Es{n}$ $R$-algebra.
\end{cor}

\begin{defn}
  Given an $\Es{n}$ $R$-algebra $A$, we define $\Es{n}$-spaces
  $\Pic(R)_{\downarrow A}$ and $BGL_1R_{\downarrow A}$ by requiring the following squares
  to be pullbacks of $\Es{n}$-monoidal categories:
  \[ \xymatrix{
    BGL_1R_{\downarrow A} \ar[r] \ar[d] & \Pic(R)_{\downarrow A} \ar[r] \ar[d] & (\Mod_R)_{/A} \ar[d] \\
    BGL_1R \ar[r] & \Pic(R) \ar[r] & \Mod_R, \\
  }\]
  where $(\Mod_R)_{/A}$ is the slice of $\Mod_R$ over the underlying
  $R$-module of $A$.
\end{defn}

Alternatively, we could define $BGL_1R_{\downarrow A}$ and
$\Pic(R)_{\downarrow A}$ merely as $\infty$-groupoids by requiring the
above diagram to consist of pullback squares of $\infty$-categories.
They would then inherit $\Es{n}$-monoidal structures from
$(\Mod_R)_{/A}$ in the same way that $BGL_1R$ and $\Pic(R)$ inherit
their structure from $\Mod_R$.

We can think of the objects of $\Pic(R)_{\downarrow A}$ as invertible
$R$-modules $M$ equipped with a map $M \to A$ of $R$-modules, and of
the morphisms as commuting triangles where the arrow $M \to M'$ is an
equivalence.
  
\begin{warning}
  Our choice of notation for $\Pic(R)_{\downarrow A}$ and
  $BGL_1R_{\downarrow A}$ is meant to distinguish these categories
  from the usual slice categories $\Pic(R)_{/A}$ and $BGL_1R_{/A}$,
  which are only defined if $R \to A$ is an equivalence. Even when
  $R=A$ and all of these are defined, they differ. Indeed,
  $\Pic(R)_{/R}$ is a slice of an $\infty$-groupoid and thus
  contractible, but $\Pic(R)_{\downarrow R}$ has many components. For
  example, the component of the identity map $R \to R$ is equivalent
  to $\Pic(R)_{/R}$ and thus contractible; but the component of the
  zero map $R \to R$ is equivalent to $BGL_1R$.
\end{warning}

We can now state a characterization of the $\Es{n}$-structure on $Mf$:

\begin{thm}\label{thomup}
  Let $X$ be an $\Es{n}$-space and $f\colon X \to \Pic(R)$ be an
  $\Es{n}$-map. The $\Es{n}$-algebra structure of $Mf$ is
  characterized by the following universal property: the space of
  $\Es{n}$ $R$-algebra maps $\Map_{\Alg_R^{\Es{n}}}(Mf, A)$ is
  equivalent to the space of $\Es{n}$-lifts of $f$ indicated below:
  \[ \xymatrix{
    & \Pic(R)_{\downarrow A} \ar[d] \\
    X \ar[r]_-f \ar@{-->}[ur] & \Pic(R). } \]
\end{thm}

\begin{proof}
  \Cref{colimup} tells us directly that
  $\Map_{\Alg_R^{\Es{n}}}(Mf, A)$ is the space of lifts of
  $X \xrightarrow{f} \Pic(R) \to \Mod_R$ to a \emph{lax}
  $\Es{n}$-functor $X \to (\Mod_R)_{/A}$. Now, edges in
  $(\Mod_R)_{/A}^\otimes$ are coCartesian for the projection
  $(\Mod_R)_{/A}^\otimes \to \Es{n}^\otimes$ if and only if their
  image in $\Mod_R^\otimes$ is coCartesian. So a lax $\Es{n}$-functor
  lifting $f$ is automatically monoidal and factors through
  $\Pic(R)_{\downarrow A}$.
\end{proof}

\begin{notn}\label{if-A-is-En+1}
  If $A$ is an \emph{$\Es{n+1}$-ring spectrum under $R$}, that is an $\Es{n+1}$-ring spectrum equipped with an $\Es{n+1}$-morphism $\eta \colon R \to A$, we use $\Ind{R}{A}$ to denote the cocontinuous $\Es{n}$-monoidal functor $\Mod_R \to \Mod_A$ induced by $\eta$ (see \cite[Proposition 7.1.2.6]{HA}). Since this functor is $\Es{n}$-monoidal, there is also an induced functor $\Alg_R^{\Es{n}} \to \Alg_A^{\Es{n}}$ for which we also use the same notation.
\end{notn}

\begin{rmk}\label{under-vs-Ralg}
  Let us explain the relation between the notions of $\Es{n+1}$-ring spectrum under $R$ and $\Es{n}$ $R$-algebra. First, if $A$ is an $\Es{n+1}$-ring spectrum under $R$, we can canonically equip it with the structure of an $\Es{n}$ $R$-algebra: by \cite[Corollary 7.3.2.7]{HA}, the right adjoint of $\Ind{R}{A} \colon \Mod_R \to \Mod_A$ is \emph{lax} $\Es{n}$-monoidal, so it preserves $\Es{n}$-algebra objects. In particular, this right adjoint lets us view the unit object $A$ as an $\Es{n}$ $R$-algebra.
  
  Now, for $n=\infty$, the concepts of $\Es{\infty}$ $R$-algebra and $\Es{\infty}$-ring spectrum under $R$ actually coincide by \cite[Variant 7.1.3.8]{HA}. They do not for $n < \infty$, but an $\Es{n}$ $R$-algebra structure on $A$ still makes $A$ an $\Es{n}$-ring spectrum under $R$ \cite[Warning 7.1.3.9]{HA}, so for such an $A$ there is still an $\Es{n-1}$-monoidal functor $\Ind{R}{A}$.
\end{rmk}

As a corollary, we easily obtain another of Lewis' results
\cite[IX.7.1]{LMS}. To state it we will use the following notation.

\begin{notn}
  For $m\le n$, let $\ind{m}{n}$ be the free $\Es{n}$ $R$-algebra on an $\Es{m}$
  $R$-algebra, i.e., the left adjoint to the forgetful functor
  $\Alg_R^{\Es{n}} \to \Alg_R^{\Es{m}}$. Likewise, we write $\ind{m}{n}(X)$ for the free $\Es{n}$-space on an $\Es{m}$-space $X$.
\end{notn}

\begin{cor}\label{thomlewisuniversal}
  For any pointed map of spaces $f\colon X \to BGL_1R$, there is a
  natural equivalence
\[\ind{0}{n}(Mf) \xrightarrow{\sim} M \bar{f},\]
where $\bar{f} \colon \ind{0}{n}(X) \to BGL_1R$ is the universal map induced by the pointed map $f$. When $X$ is connected, then $\ind{0}{n}(X)$ is equivalent to $\Omega^n \Sigma^n X$ and $\bar{f}$ is the $n$-fold loop map adjoint to the pointed map $\Sigma^nf$.
\end{cor}

\begin{proof}
  We will check that both sides represent the same functor on
  $\Alg_R^{\Es{n}}$; this yields the claim. Let $A$ be an arbitrary $\Es{n}$ $R$-algebra. By adjunction, we
  have that
  $\Map_{\Alg_R^{\Es{n}}}(\ind{0}{n}(Mf), A) \cong
  \Map_{\Alg_R^{\Es{0}}}(Mf,A)$
  and by \Cref{thomup}, this is the space $L_0$ of lifts of $f$ to an
  $\Es{0}$-map (i.e., a pointed map) $X \to BGL_1R_{\downarrow A}$. Moreover, and again
  by \Cref{thomup}, $\Map_{\Alg_r^{\Es{n}}}(M\bar{f},A)$ is the space
  $L_n$ of lifts of $\bar{f}$ to an $\Es{n}$-map
  $\ind{0}{n}(X) \to BGL_1R_{\downarrow A}$.

  Since $\ind{0}{n}(X)$ is the free $\Es{n}$-space on the
  pointed space $X$, the solid vertical maps (induced by
  composition with the unit $X \to \ind{0}{n}(X)$) in the
  following diagram of fiber sequences are equivalences:
  \[ \xymatrix{
    L_n \ar[r] \ar@{-->}[d] & \Map_{\Es{n}}(\ind{0}{n}(X), BGL_1R_{\downarrow A}) \ar[r] \ar[d]^\sim &  \Map_{\Es{n}}(\ind{0}{n}(X), BGL_R) \ar[d]^\sim \\
    L_0 \ar[r] & \Map_{\Es{0}}(X, BGL_1R_{\downarrow A}) \ar[r] &
    \Map_{\Es{0}}(X, BGL_1R). }\]
  It follows that the induced map $L_0 \to L_n$ is an equivalence, too.
\end{proof}

\begin{rmk}
  This result also effortlessly extends to more general operads in
  place of $\Es{n}$, which in fact is the version proved by Lewis. We
  can replace $\Es{0}$ as well; for example, the same argument as used
  above also proves that if $f$ is an $m$-fold loop map and $\bar{f}$
  denotes the universal extension to an $n$-fold loop for $n \ge m$,
  then $\ind{m}{n}(Mf) \xrightarrow{\sim} M \bar{f}$.
\end{rmk}

\subsection{$\Es{n}$-orientations}

The Thom spectrum $MG$ for a topological group $G \to O$ arises
classically in the theory of orientations of vector bundles,
representing the universal cohomology theory that orients manifolds
with structure group $G$. This point of view admits a generalization
to $\Es{n}$-ring spectra, as we briefly summarize now; see
\cite{ABGHRinfcat} and \cite{ABGHRunits} for a comparison of the
various notions of orientations in the $\Es{1}$ and $\Es{\infty}$
cases, as well as~\cite{Umkehr}.

For the remainder of this section let $R$ be an $\Es{n+1}$-ring
spectrum and $A$ be an $\Es{n+1}$-ring spectrum under $R$ (see
\Cref{if-A-is-En+1}).

\begin{conv}\label{0-fold-loop}
  For $n=0$ we make special arrangments: by \emph{grouplike
    $\Es{0}$-space} we mean a connected pointed space, and by
  \emph{$0$-fold loop map} we mean a pointed map with connected
  domain.
\end{conv}

\begin{defn} Let $B(R,A)$ be the full subgroupoid of $\Pic(R)_{\downarrow A}$
  consisting of morphisms of $R$-modules $h \colon M \to A$ such that
  the adjoint $\adjoint{h} \colon \Ind{R}{A}(M) \to A$ is an
  equivalence.
\end{defn}

To study the relation between $B(R,A)$ and orientations, we need a
lemma about these $R$-module morphisms:

\begin{lemma}\label{A-eqs-both-ways}
  Let $h_i \colon M_i \to A$ for $i=1,2$ be two morphisms of
  $R$-modules. Then $\adjoint{(h_1 \otimes_{\downarrow A} h_2)}$ is an
  equivalence of $A$-modules if and only if both $\adjoint{h_1}$ and
  $\adjoint{h_2}$ are (here $\otimes_{\downarrow A}$ denotes the tensor product
  in $\Pic(R)_{\downarrow A}$).
\end{lemma}

\begin{proof}
  Recall that in the monoidal structure of $\Pic(R)_{\downarrow A}$,
  $h_1 \otimes_{\downarrow A} h_2$ is given by the composite
  $M \otimes_R N \xrightarrow{h_1 \otimes_R h_2} A \otimes_R A
  \xrightarrow{\mu_A} A$, where $\mu_A$ is the multiplication on $A$.
  Since $\Ind{R}{A}$ is $\Es{n}$-monoidal, we see that
  $\adjoint{(h_1 \otimes_{\downarrow A} h_2)} = \adjoint{h_1} \otimes_A
  \adjoint{h_2}$.

  This makes it clear that if both $\adjoint{h_i}$ are equivalences,
  then so is $\adjoint{(h_1 \otimes_{\downarrow A} h_2)}$.

  Now assume $\adjoint{(h_1 \otimes_{\downarrow A} h_2)}$ is an
  equivalence with inverse $g$ and let us prove that both
  $\adjoint{h_i}$ are equivalences. To lighten the notation, set
  $\adjoint{M_i} := \Ind{R}{A}(M_i)$. Notice that by definition of $g$,
  the morphism of $A$-modules given by the composite
  \[ A \xrightarrow{g} \adjoint{M_1} \otimes_A \adjoint{M_2}
    \xrightarrow{\id_{\adjoint{M_1}} \otimes_A \adjoint{h_2}} \adjoint{M_1}
    \otimes_A A\]
  is a section of $\adjoint{h_1}$, which shows that $\adjoint{M_1}$ splits
  as $A \oplus F_1$ where $F_1 := \fib(\adjoint{h_1})$. Under this
  splitting, $\adjoint{h_1}$ corresponds to the projection $A \oplus F_1
  \to A$. Of course there is an analogous splitting of $\adjoint{M_2}$
  as $A \oplus F_2$.

  Then we get that $\adjoint{h_1} \otimes_A \adjoint{h_2}$ is the
  projection of $(A \oplus F_1) \otimes_A (A \oplus F_2) = A \oplus
  F_1 \oplus F_2 \oplus F_1 \otimes F_2$ onto the first summand. Since
  this map is assumed to be an equivalence, we conclude $F_1 = F_2 = 0$
  and thus both $\adjoint{h_1}$ and $\adjoint{h_2}$ are equivalences.
\end{proof}

As a corollary of (one direction) of the lemma, the set of objects in
$B(R,A)$ is closed under tensor products in $\Pic(R)_{\downarrow A}$, and thus
$B(R,A)$ inherits an $\Es{n}$-monoidal structure.

We can now define orientations analogously to the definitions in
\cite{ABGHRinfcat} and \cite{ABGHRunits} for the $\Es{1}$ and
$\Es{\infty}$ cases.

\begin{defn}
  The \emph{space of $\Es{n}$ $A$-orientations} of an $\Es{n}$-map $f
  \colon X \to \Pic(R)$ is the space of
  $\Es{n}$-lifts of $f$ indicated below:
  \[\xymatrix{& B(R,A) \ar[d] \\
      X \ar[r] \ar@{-->}[ur] & \Pic(R) }\]
\end{defn}

Notice that because $B(R,A)$ is an $\Es{n}$-subspace of
$\Pic(R)_{\downarrow A}$, the universal property (\Cref{thomup}) implies that an
$\Es{n}$ $A$-orientation of $f$ determines a map of $\Es{n}$
$R$-algebras $Mf \to A$. Of course, not every such map of algebras is
an orientation; one way to state the requirement of factoring through
$B(R,A)$ is to say that for every point $x \colon \ast \to X$, the
adjoint $\theta_x^\dagger$ of the $R$-module map
$\theta_x \colon M(f \circ x) \to Mf \to A$ is an equivalence. But if
the $\Es{n}$-space $X$ is grouplike, then this condition is
automatically satisfied, as we show next.

\begin{lemma}\label{algmap-is-orient}
  If $f : X \to \Pic(R)$ is an $n$-fold loop map, then an $\Es{n}$
  $R$-algebra morphism $Mf \to A$ is automatically an orientation and
  the space of such algebra morphisms is equivalent to the space of
  $\Es{n}$ $A$-orientations of $f$.
\end{lemma}

\begin{proof}
  We just need to show that the lifts of $f\colon X \to \Pic(R)$ to
  $\Pic(R)_{\downarrow A}$ as considered in \Cref{thomup} factor through
  $B(R,A)$ (we defined $B(R,A)$ as the union of some of the connected
  components of $\Pic(R)_{\downarrow A}$). Such a lift gives for each point
  $x \in X$ an $R$-linear map $\alpha_x \colon f(x) \to A$.

  If $n>0$, we have at least one multiplication on $X$ and inverses
  for it, so the map $\alpha_{x^{-1}} \colon f(x^{-1}) \to A$
  satisfies that
  $\alpha_{x} \otimes \alpha_{x^{-1}} \simeq \alpha_{1_X} \colon
  f(1_X) \to A$ is homotopic to the unit $R \to A$ and thus is an
  equivalence when induced up to $A$. By \Cref{A-eqs-both-ways}, this
  implies $\alpha_x$ is in $B(R,A)$.

  If $n=0$, we don't even need \Cref{A-eqs-both-ways}: $X$ is
  connected by \Cref{0-fold-loop} so we can pick a path from
  $x$ to $1_X$ and the lift of $f$ gives a corresponding equivalence
  $f(x) \simeq f(1_X)$ that commutes up to homotopy with the maps
  $\alpha_{x}$ and $\alpha_{1_X}$ to $A$.
\end{proof}

As expected $\Es{n}$ $A$-orientations give rise to Thom isomorphisms
of $\Es{n}$ algebras.

\begin{prop}\label{thomiso}
  An $\Es{n}$ $A$-orientation of an $\Es{n}$-map
  $f \colon X \to \Pic(R)$ give rise to a \emph{Thom isomorphism}
  $\Ind{R}{A}(Mf) \cong \Ind{S}{A}(\Sigma^\infty_+ X)$ of $\Es{n}$
  $A$-algebras (where $S$ is the sphere spectrum). The equivalence of
  $R$-modules underlying this Thom isomorphism is a map
  $A \otimes_R Mf \to A \otimes \Sigma^\infty_+ X$.
\end{prop}

\begin{proof}
  We can also describe $B(R,A)$ as a pullback of $\Es{n}$-spaces:
  \[\xymatrix{
      B(R,A) \ar[r] \ar[d] & B(A,A) \ar[d] \\
      \Pic(R) \ar[r]_{\Ind{R}{A}} & \Pic(A) }\]

  Notice that from the definition, $B(A,A)$ is the space of
  trivialized $A$-modules denoted by $A\text{-triv}$ in
  \cite{ABGHRinfcat}. Since it is the (entire) slice of $\Pic(A)$ over
  $A$, it is contractible. Thus a lift of $f$ to $B(R,A)$ gives a
  null-homotopy of $\Ind{R}{A} \circ f$, and consequently,
  $\Ind{R}{A} \circ f$ has the same Thom spectrum as the null map
  $X \to \Pic(A)$.

  Since $\Ind{R}{A} \colon \Mod_R \to \Mod_A$ is a left adjoint, it
  preserves the colimit computing the Thom spectrum of $f$, that is,
  $M(\Ind{R}{A} \circ f) \cong \Ind{R}{A}(Mf)$. On the other hand the
  null map can be expressed as $\Ind{R}{S} \circ c$ where
  $c \colon X \to \Pic(S)$ is null. The same argument shows that
  $M(\Ind{S}{A} \circ c) \cong \Ind{S}{A}(\colim_X S) =
  \Ind{S}{A}(\Sigma^\infty_+ X)$.

  Finally, the underlying $R$-module spectrum of $\Ind{R}{A}(M)$ is
  given by $Mf \otimes_R A$.
\end{proof}

\begin{cor}\label{Mf-orientable}
  Let $n > 0$ and $f\colon X \to \Pic(R)$ be an $n$-fold loop map with Thom spectrum $Mf$. Then $f$ is canonically $\Es{n-1}$ $Mf$-orientable.
\end{cor}

\begin{proof}
  Recall that $Mf$ is an $\Es{n}$ $R$-algebra, which we can regard as an $\Es{n}$-ring spectrum under $R$ (see the second paragraph of \Cref{under-vs-Ralg}). The identity morphism $Mf \to Mf$ gives the required orientation.
\end{proof}

The corresponding Thom isomorphism is an equivalence of $\Es{n-1}$ $R$-algebras, $\Ind{R}{Mf}(Mf) \cong \Ind{S}{Mf}(\Sigma^\infty_+ X)$ whose underlying $R$-module equivalence is a morphism $Mf \otimes_R Mf \cong Mf \otimes \Sigma^\infty_+ X$. This Thom isomorphism is due to Mahowald, see \cite[Theorem 1.2]{MahowaldRS}, see also  \cite[Cor.~1.8]{Umkehr}

\smallskip

Now we give an alternative proof and mild generalization of the description of $\Es{n}$-orientations of ring spectra due to Chadwick and Mandell, see~\cite[Theorem 3.2]{engenera}. While their argument uses the general Thom isomorphism, we deduce this result directly from the universal property of Thom ring spectra.

\begin{cor}
  Let $f\colon X \to \Pic(R)$ be an $n$-fold loop map and let $A$ be an $\Es{n+1}$-ring spectrum under $R$. The space of $\Es{n}$ $R$-algebra maps from $Mf$ to $A$ (regarded as an $\Es{n}$ $R$-algebra) is either empty or equivalent to the space of $\Es{n}$ $R$-algebra maps from $R \otimes \Sigma^{\infty}_+X$ to $A$, that is
  \[\Map_{\Alg_R^{\Es{n}}}(Mf,A) \simeq \Map_{\Alg_R^{\Es{n}}}(R
    \otimes \Sigma^{\infty}_+X,A).\]
\end{cor}

\begin{proof}
  In view of the pullback square of $\Es{n}$-spaces
  \[\xymatrix{
      \Pic(R)_{\downarrow A} \ar[r] \ar[d] & \Pic(A)_{\downarrow A} \ar[d] \\
      \Pic(R) \ar[r]_{\Ind{R}{A}} & \Pic(A), }\]
  the space of $\Es{n}$-lifts of $f\colon X \to \Pic(R)$ to $\Pic(R)_{\downarrow A}$ is equivalent to the space of $\Es{n}$-lifts of $\Ind{R}{A} \circ f$ to $\Pic(A)_{\downarrow A}$. Assume $\Map_{\Alg_R^{\Es{n}}}(Mf,A)$ is non-empty. Any $\Es{n}$ $R$-algebra map in that space is automatically an orientation by \Cref{algmap-is-orient} and thus gives a homotopy between $\Ind{R}{A} \circ f$ and the constant map $c_A \colon X \to \Pic(A)$ through $\Es{n}$-maps. That homotopy guarantees the spaces of $\Es{n}$-lifts of $\Ind{R}{A} \circ f$ and $c_A$ are equivalent, and the universal property (\Cref{thomup}) says these spaces of lifts are the spaces of algebra maps in the statement.
\end{proof}

A similar result is true for spaces of orientations with a very similar proof.

\begin{cor}
  Let $f\colon X \to \Pic(R)$ be an $\Es{n}$-map and let $A$ be an $\Es{n+1}$-ring spectrum under $R$. The space of $\Es{n}$ $A$-orientations for $f$ is either empty or equivalent to the space of $\Es{n}$ $A$-orientations of the constant map $c_R \colon X \to \Pic(R)$, namely $\Omega \Map_{\Es{n}}(X, \Pic(A))$.
\end{cor}

\begin{proof}
  Arguing just as for the previous corollary using $B(-,A)$ in place of $\Pic(-)_{\downarrow A}$ yields everything but the concrete description of the space of orientations for the constant map. To see that, notice that unlike $\Pic(A)_{\downarrow A}$, the space $B(A,A) \simeq \Pic(A)_{/A}$ is contractible. This says first that $B(R,A)$ is the fiber of $\Ind{R}{A}\colon \Pic(R) \to \Pic(A)$ and second that the space of $\Es{n}$-lifts of the null map $c_R$ to $B(R,A)$ is equivalent to $\Map_{\Es{n}}(X, \Omega\Pic(A)) \simeq \Omega \Map_{\Es{n}}(X, \Pic(A))$.
\end{proof}

\section{Thom spectra and versal $\Es{n}$-algebras}\label{sec:char}

\subsection{Characteristics of ring spectra}

Let $0 \le n \le \infty$ be an integer or $\infty$ and consider an $\Es{n+1}$-ring
spectrum $R$ with associated category $\Mod_R$ of left $R$-module. As explained above, this
category admits a monoidal product $\otimes = \otimes_R$ equipping it
with the structure of an $\Es{n}$-monoidal category, so that there
is a category $\Alg_R^{\Es{n}}$ of $\Es{n}$-algebras over $R$.

We now extend the definition of characteristic introduced in \cite{szymikcharacteristics1,szymikcharacteristics2} in two ways: Firstly, we allow arbitrary homotopy classes $S^k \to R$ in non-negative degrees and secondly we consider $\Es{n}$-algebras. For background and further results in the $\Es{\infty}$ case we refer to Szymik's papers.

\begin{defn}\label{defnchar}
  Given $k\ge 0$ and $\chi \in \pi_kR$ thought of as an $R$-linear map
  $\Sigma^k R \to R$, an algebra $A \in \Alg_R^{\Es{n}}$ with unit
  $\eta\colon R \to A$ is said to be of \emph{characteristic} $\chi$
  if $\eta \circ \chi \colon \Sigma^k R \to A$ is null-homotopic.
\end{defn}

\begin{rmk}
If $n\ge 1$ and $A \in \Alg_R^{\Es{n}}$ is of characteristic $p$, then the  existence of the multiplication map $\mu\colon A \otimes_R A \to A$ shows $p \cdot \id_A =0$. However, note that for $n=0$ an algebra being of characteristic $p$ does not imply $p \cdot \id_A$ is null, as the mod 2 Moore spectrum demonstrates. 
\end{rmk}

\begin{defn}
  For a given $\chi\colon \Sigma^k R \to R$, we define the \emph{versal
  $R$-algebra $\modmod[\Es{n}]{R}{\chi}$ of characteristic} $\chi$ as
  the following pushout in $\Alg_R^{\Es{n}}$:
  \[\xymatrix{\free{n} \Sigma^k R \ar[r]^-{\bar{0}} \ar[d]_{\bar{\chi}} & R \ar[d] \\
    R \ar[r] & \modmod[\Es{n}]{R}{\chi},}\]
  where
  \begin{itemize}
  \item $F_{\Es{n}} \colon \Mod_R \to \Alg_R^{\Es{n}}$ is the
    free $\Es{n}$ $R$-algebra functor, the left adjoint
    to the forgetful functor $\Alg_R^{\Es{n}} \to \Mod_R$.
  \item $\bar{g}\colon \free{n}M \to N$ denotes the algebra map
    which is adjoint to the $R$-linear map $M \to N$ with
    $N \in \Alg_R^{\Es{n}}$.
  \end{itemize}
We use the term \emph{versal} rather than \emph{universal} as these algebras are not initial but only weakly initial in general, see \cite[Prop. 3.11]{szymikcharacteristics1}. When there is no risk of confusion, we shall write
  $\modmod{R}{\chi}$ instead of $\modmod[\Es{n}]{R}{\chi}$.
\end{defn}

The versal $\Es{n}$ $R$-algebra $\modmod{R}{\chi}$ admits a
different characterization as the free $\Es{n}$ $R$-algebra on the
\emph{pointed spectrum}
$R\to R/\chi = \cof(\Sigma^k R \xrightarrow{\chi} R)$, that is,
$\modmod[\Es{n}]{R}{\chi} =
\ind{0}{n}\left(\modmod[\Es{0}]{R}{\chi}\right)$. More generally we have:

\begin{lemma}
Let $0\le m\le n$, then the following $\Es{n}$ $R$-algebras are equivalent:
\begin{enumerate}
 \item The versal characteristic $\chi$ algebra $\modmod[\Es{n}]{R}{\chi}$,
 \item the free $\Es{n}$ $R$-algebra $\ind{m}{n}(\modmod[\Es{m}]{R}{\chi})$ on the $\Es{m}$ R-algebra $\modmod[\Es{m}]{R}{\chi}$, 
 \item the coequalizer $C_R(\chi)=\coeq\left(\free{n}\Sigma^kR \overset{\bar{\chi}}{\underset{\bar{0}}{\rightrightarrows}} R\right)$ taken in $\Alg_R^{\Es{n}}$.
\end{enumerate}
\end{lemma}
\begin{proof}
We first prove the equivalence of the algebras in (1) and (2). On the one hand, applying the left adjoint $\ind{m}{n}$ to the defining pushout diagram for $\modmod[\Es{m}]{R}{\chi}$ gives a pushout square
\[\xymatrix{\ind{m}{n} \free{m} \Sigma^k R \ar[r]^-{\bar{0}} \ar[d]_{\ind{m}{n} \bar{\chi}} & \ind{m}{n} R \ar[d] & \\
\ind{m}{n} R \ar[r] & \ind{m}{n}(\modmod[\Es{m}]{R}{\chi}).}\]
On the other hand, since $\ind{m}{n}$ preserves the initial $R$-algebra $R$ and $\ind{m}{n} \free{m} \simeq \free{n}$, this diagram is naturally equivalent to the pushout square
\[\xymatrix{\free{n} \Sigma^k R \ar[r]^-{\bar{0}} \ar[d]_{\bar{\chi}} & R \ar[d] \\
    R \ar[r] & \modmod[\Es{n}]{R}{\chi},}\]
hence it follows immediately that $\modmod[\Es{n}]{R}{\chi}$ is equivalent to $\ind{m}{n}(\modmod[\Es{m}]{R}{\chi})$.

To prove that $\modmod[\Es{n}]{R}{\chi}$ is equivalent to the coequalizer $C_R(\chi)$, we observe that both algebras corepresent the same functor on $\Alg_R^{\Es{n}}$, which maps a given algebra $A$ with unit map $\eta$ to the space of null-homotopies of the composite $\Sigma^k R \xrightarrow{\chi} R \xrightarrow{\eta} A$. 
\end{proof}

There is a Thom-isomorphism type lemma for algebras of a given characteristic, which is a straightforward generalization of~\cite[Prop. 3.2]{szymikcharacteristics1}.

\begin{lemma}\label{charthomeq}
  Let $n\ge 1$ and suppose $\eta \colon R \to A$ is a map of
  $\Es{n+1}$-ring spectra where $A$ has characteristic $\chi$, then
  there is a natural equivalence of $\Es{n}$ $A$-algebras
  \[\Ind{R}{A}(\free{n}(\Sigma^{k+1} R)) \xrightarrow{\;\sim\;}
  \Ind{R}{A}(\modmod[\Es{n}]{R}{\chi}).\]
\end{lemma}

\begin{proof}
  The two algebras in question are constructed as pushouts in
  $\Alg_{R}^{\Es{n}}$
\[\xymatrix{\free{n}\Sigma^kR \ar[r]^{\bar{0}} \ar[d]_{\bar{0}} & R \ar[d] & \free{n}\Sigma^k R \ar[r]^{\bar{0}} \ar[d]_{\bar{\chi}} & R \ar[d]\\ 
R \ar[r] & \free{n}(\Sigma^{k+1} R) & R \ar[r] & \modmod{R}{\chi}.}\]
Since
$\free{n}(\Sigma^kR) \otimes_R A \xrightarrow{\chi \otimes \id_A} A$
is homotopic to $0$ if $A$ has characteristic $\chi$, the two diagrams
become equivalent after applying the functor
$\Ind{R}{A} \colon \Alg^{\Es{n}}_R \to \Alg^{\Es{n}}_A$ (see
\Cref{if-A-is-En+1}), which is given on the level of $R$-modules by
smashing with $A$.
\end{proof}

\begin{lemma}\label{charpmappinguniversal}
  Suppose $A \in \Alg_R^{\Es{n}}$. If $A$ is of characteristic
  $\chi\colon \Sigma^kR \to R$, then
  \[\Map_{\Alg_R^{\Es{n}}}(\modmod[\Es{n}]{R}{\chi},A) =
  \Omega^{\infty + k+1}A;\]
  otherwise the mapping space is empty.
\end{lemma}

\begin{proof}
  The proof of \cite[Cor. 2.9]{szymikcharacteristics1} generalizes
  effortlessly: By definition of $\modmod{R}{\chi}$ as an
  $\Es{n}$-algebra, there exists a pullback square
  \[\xymatrix{\Map_{\Alg_R^{\Es{n}}}(\modmod{R}{\chi},A) \ar[r] \ar[d] &
  \Map_{\Alg_R^{\Es{n}}}(R,A)= \ast \ar[d]^{\bar{0}^*} \\
  \ast = \Map_{\Alg_R^{\Es{n}}}(R,A) \ar[r]_-{\bar{\chi}^*} &
  \Map_{\Alg_R^{\Es{n}}}(\free{n}\Sigma^{k}R,A) =
  \Omega^{\infty+k}A}\]
showing that
$\Map_{\Alg_R^{\Es{n}}}(\modmod{R}{\chi},A) \xrightarrow{\sim}
\Omega^{\infty +k+1}A$
if $A$ is of characteristic $\chi$ so that the two points lie in the
same component of $\Omega^{\infty}A$, and empty otherwise.
\end{proof}

In other words, specializing to the case in which $\chi$ is the
multiplication by $p$ map we see that
$\Map_{\Alg_R^{\Es{n}}}(\modmod{R}{p},A)$ is the space of null-homotopies of the composite map
$R \xrightarrow{p} R \xrightarrow{\eta} A$ in $R$-modules, which is
equivalent to the space of homotopies $\eta \sim (1-p)\eta$ under the
map $\Hom_R(R,A) \to \Hom_R(R,A)$ given by
$\alpha \mapsto \alpha + \eta$. This in turn admits an interpretation
in terms of Thom spectra, given in the next section.

\subsection{Thom spectra as versal characteristic $\chi$ algebras}

The goal of this section is to identify the Thom spectrum classified by a map $f\colon S^{k+1} \to BGL_1R$ with the versal characteristic $\chi$ algebra, where $\chi=\chi(f)\colon \Sigma^kR \to R$ is the characteristic corresponding to $f$ as defined below.

\begin{lemma}\label{pblemma}
  If $k\ge 0$ and $A$ is any unital $R$-module, that is, an $R$-module
  equipped with an $R$-linear map $\eta \colon R \to A$, then there is a
  natural pullback square of spaces
  \[\xymatrix{\Map_*(S^{k+1}, BGL_1R_{\downarrow A}) \ar[r] \ar[d] & \ast \ar[d]^{\underline{\eta}} \\
    \Map_*(S^{k+1},BGL_1R) \ar[r] & \Map_*(S^{k},\Omega^{\infty}A) =
    \Omega^{\infty+k}A}\]
  where the right vertical map picks out the constant map with value the unit 
  $\eta \in \Hom_R(R,A) = \Omega^{\infty}A$.
\end{lemma}

\begin{proof}
  If we regard $GL_1R$ as the space of $R$-linear self-equivalences of
  $R$, it acts on $\Hom_R(R,A)$ by pre-composition. This induces a map
  $\eta_\ast \colon GL_1R \to \Hom_R(R,A)$ defined by
  $g \mapsto \eta \circ g$, where $\eta \colon R \to A$ is the unit of
  $A$. Let $GL_1R_{\downarrow A}$ be the fiber of this map at the point
  $\eta \in \Hom_R(R,A)$, so that $GL_1R_{\downarrow A}$ can be thought of as
  the space of automorphisms $g$ of $R$ compatible with $\eta$ --- in
  the sense that $\eta \circ g$ is homotopic to $\eta$. This
  $GL_1R_{\downarrow A}$ can alternatively be described as the space of
  endomorphisms of the object $\eta$ in the $\infty$-groupoid
  $BGL_1R_{\downarrow A}$. In other words,
  $\Omega_\eta(BGL_1R_{\downarrow A}) = GL_1R_{\downarrow A}$.

  This shows that $\Map_*(S^{k},GL_1R) = \Map_*(S^{k+1}, BGL_1R)$ and
  similarly for $GL_1R_{\downarrow A}$. Therefore, applying $\Map_\ast(S^k, -)$
  to the pullback square
  \[\xymatrix{GL_1R_{\downarrow A} \ar[r] \ar[d] & \ast \ar[d]^{\eta} \\
    GL_1R \ar[r]_-{\eta_\ast} & \Hom_R(R,A)}\]
  gives the desired result.
\end{proof}

\begin{defn}\label{defassocchar}
Suppose given a map $f\colon S^{k+1} \to BGL_1R$ and let $\tilde{f}\colon \Sigma^kR \to R$ be the associated homotopy class. The \emph{characteristic $\chi(f)\colon \Sigma^kR \to R$ associated with $f$} is then defined by
\[\chi(f) = 
\begin{cases}
\tilde{f} -1 & \text{if } k=0 \\
\tilde{f}     & \text{if } k>0.
\end{cases}\]
\end{defn}

\begin{prop}\label{thommappinguniversal}
If $f\colon S^{k+1} \to BGL_1R$ is a based map and $\bar{f} \colon \Omega^n \Sigma^n S^{k+1} \to BGL_1R$ is the corresponding $n$-fold loop map, then for any $A \in \Alg_R^{\Es{n}}$, there is an equivalence of spaces
\[\Map_{\Alg_R^{\Es{n}}}(M\bar{f},A) = \Omega^{\infty + k+1}A \]
if $A$ has characteristic $\chi(f)$; otherwise, the mapping space is empty. 
\end{prop}
\begin{proof}
By \Cref{thomlewisuniversal}, it suffices to prove that $\Alg_R^{\Es{0}}(Mf,A) = \Omega^{\infty + k+1}A$. To this end, let $A$ be a unital $R$-module and consider the following commutative diagram of spaces
\[\xymatrix{\Alg^{\Es{0}}_R(Mf,A) \ar[r] \ar[d] & \Map_*(S^{k+1}, BGL_1R_{\downarrow A}) \ar[r] \ar[d] & \ast \ar[d]^{\underline{\eta}} \\
\ast \ar[r]^-f \ar@/^-1.4pc/[rr]_{\eta_* \circ f} & \Map_*(S^{k+1},BGL_1R) \ar[r] & \Map_*(S^{k},A) = \Omega^{\infty+k}A.}\]

The left square is a pullback by the universal property of Thom spectra (\Cref{thomup}), while the right square is one due to \Cref{pblemma}. Hence the outer rectangle is a pullback square. 

The right vertical arrow $\underline{\eta}$ is induced by $\ast \to \Omega^{\infty}A$ corresponding to the unit $\eta$ of $A$, so
\[[\underline{\eta}] = 
\begin{cases}
\eta \in \pi_0(\Omega^{\infty}A) & \text{if } k=0 \\
0 \in \pi_k(\Omega^{\infty}A)      & \text{if } k>0.
\end{cases}\]
This explains the distinction of the two cases appearing in \Cref{defassocchar}: the pullback is non-empty if $k=0$ and $\eta_* f \simeq \eta$, or if $k>0$ and $\eta_* f \simeq 0$.
\end{proof}

We thus obtain:

\begin{thm}\label{charpthomcomparison}
For any $n\ge 0$ and $f\colon S^{k+1} \to BGL_1R$ with corresponding $n$-fold loop map $\bar{f} \colon \Omega^n \Sigma^n S^{k+1} \to BGL_1R$ and with associated characteristic $\chi = \chi(f)$, there is an equivalence $M\bar{f} = \modmod[\Es{n}]{R}{\chi}$ of $\Es{n}$ $R$-algebras.
\end{thm}
\begin{proof}
It follows from \Cref{charpmappinguniversal} that $M\free{n}{f}$ is of characteristic $\chi$, so the universal property of $\modmod{R}{\chi}$ gives a map $\modmod{R}{\chi} \to M\free{n}{f}$ of $\Es{n}$ $R$-algebras. In view of \Cref{thommappinguniversal} and \Cref{charpmappinguniversal}, $M\free{n}{f}$ and $\modmod{R}{\chi}$ corepresent the same functor on $\Alg_{R}^{\Es{n}}$, hence this map is an equivalence.
\end{proof}

\begin{rmk}
An alternative proof of \Cref{charpthomcomparison} starts with \Cref{thomlewisuniversal} to first reduce the claim to the $\Es{0}$ case, and then identifies the Thom spectrum with the versal characteristic $\chi$ $\Es{0}$-algebra. The latter statement admits a direct and easy computational proof, as can be found for example in \cite[Lemma 3.3]{TEMSS} for the case $f=(p-1)$.
\end{rmk}

\section{Applications}\label{sec:applications}

\subsection{The Hopkins--Mahowald theorem I: $H\F_p$}\label{sec:hfp}

In this section we will assume that all spaces and spectra are implicitly $p$-complete for some prime $p$. Let $f_p\colon S^1 \to BGL_1(S^0)$ be the map corresponding to the element $1-p \in \Z_p^{\times} \cong \pi_1BGL_1(S^0)$ and $\bar{f}_p$ the extension of $f_p$ to a double loop map, making the following diagram commute:
\[\xymatrix{S^1 \ar[r]^-{f_p} \ar[d] & BGL_1S^0 \\
\Omega^2\Sigma^2S^1. \ar@{-->}[ru]_{\bar{f}_p}}\]

\begin{thm}\label{hopkinsmahowaldthm}
If $f_p\colon S^1 \to BGL_1(S^0)$ is the map corresponding to the element $1-p \in \Z_p^{\times} \cong \pi_1BGL_1(S^0)$, then the following three spectra are equivalent as $\Es{2}$-algebras:
\begin{enumerate}
 \item The Thom spectrum $M\bar{f}_p$, 
 \item the versal characteristic $p$ $\Es{2}$-algebra $\modmod[\Es{2}]{S^0}{p}$, and
 \item the Eilenberg--Mac Lane spectrum $H\F_p$ viewed as an $\Es{2}$-algebra. 
\end{enumerate}
\end{thm}
\begin{proof}
The equivalence of $M\bar{f}_p$ and $\modmod{S^0}{p}$ as $\Es{2}$-algebras is a special case of \Cref{charpthomcomparison}. In order to prove that these  algebras are also equivalent to $H\F_p$ viewed as an $\Es{2}$-algebra, we observe that there exists a (canonical) map of connective $\Es{2}$-algebras
\[\xymatrix{\phi\colon\modmod{S^0}{p} \ar[r] & H\F_p}\]
witnessing the homotopy $p \sim 0$,
because $H\F_p$ is of characteristic $p$ and $\pi_1H\F_p=0$. Taking $A=H\F_p$ in  \Cref{charthomeq} and using the computation of the homology of free $\Es{2}$-algebras by Araki and Kudo~\cite[Thm.~7.1]{arakikudo} for $p=2$ and by Dyer and Lashof~\cite[Thm.~5.2]{dyerlashof} for $p>2$  shows that 
\[\pi_*(\modmod{S^0}{p} \otimes H\F_p) \cong \pi_*(\free{2}S^1 \otimes H\F_p) \cong H_*(\free{2}S^1,\F_p) \cong \cA_p,\]
the dual of the mod $p$ Steenrod algebra. Moreover, the $H\F_p$ $\Es{2}$-algebra morphism
\[\xymatrix{\free{2}S^1 \otimes H\F_p \ar[r]^-{\sim} & \modmod{S^0}{p} \otimes H\F_p \ar[r]^{\phi \otimes H\F_p} & H\F_p \otimes H\F_p}\] 
is adjoint to $S^1 \to H\F_p \otimes H\F_p$. Unwinding the construction, we see that this map picks out the Bockstein, so the equivalence of (2) and (3) follows from Steinberger's computation of the Dyer--Lashof operations~\cite[Ch.~3, Thms.~2.2, 2.3]{hinfty} on both sides.
\end{proof}

\begin{rmk}
The equivalence between the versal $\Es{2}$-algebra of characteristic $p$ and $H\F_p$ has been proven independently in~\cite{mnnnilpotence}.
\end{rmk}

\begin{rmk}
The reader might wonder why the above argument does not apply to the unique (up to contractible choice) map $\psi\colon \free{2}S^1 \to H\F_p$ to imply that $\free{2}S^1 \simeq \modmod{S^0}{p}$. This conclusion is false since the two algebras corepresent different functors: $\free{2}S^1$ corepresents $\Omega^{\infty+1}$, while the functor corepresented by $\modmod{S^0}{p}$ only agrees with that on $\Es{2}$-algebras of characteristic $p$ by \Cref{charpmappinguniversal}.

The reason the argument does not apply is that the algebra map $\psi \otimes \F_p\colon \free{2}S^1 \otimes H\F_p \to H\F_p \otimes H\F_p$ does not correspond to a non-zero multiple of the Bockstein, so it is not an equivalence. This also shows that, in general, the equivalence in \Cref{charthomeq} does not arise from a morphism of algebras $\free{n}(\Sigma^{k+1} R) \to \modmod{R}{\chi}$. 
\end{rmk}

The Hopkins--Mahowald theorem has the following well-known application. Recall that the Morava $K$-theory spectrum $K(n)$ of height $n$ is an $\Es{1}$-ring spectrum with coefficients $K(n)_* = \F_{p^n}[v_n^{\pm 1}]$, where the degree of $v_n$ is $2p^{n}-2$. 

\begin{cor}
There is no $\Es{2}$-refinement of the $\Es{1}$-ring structure on $K(n)$. 
\end{cor}
\begin{proof}
Assume that $K(n)$ has the structure of an $\Es{2}$-ring spectrum. Since $p$ acts trivially on the coefficients, the universal property of $H\F_p$ induces a map of $\Es{2}$-ring spectra
\[
\xymatrix{H\F_p \ar[r] & K(n).}
\]
However, since $H\F_p$ is $K(n)$-acyclic for all finite $n$ by \cite[Theorem 2.1(i)]{ravconj}, this map must be null, which yields a contradiction. 
\end{proof}

\subsection{The Hopkins--Mahowald theorem II: $H\Z$}

The goal of this section is to deduce an integral analogue of \Cref{hopkinsmahowaldthm}, using the framework of this paper. Our approach is different from the arguments given in~\cite{MahowaldRS},~\cite{CMT}, and~\cite{blumbergthhthom}, as it does not rely on any further explicit homology computations. Unless otherwise stated, we will implicitly work in the $p$-complete category of spaces and spectra. 

We start with a folklore result of independent interest, which shows that the $\mathbb{E}_n$-ring structure on Eilenberg--MacLane ring spectra are essentially unique, for all $0 \le n \le \infty$. Since we do not know a published reference for this result, we sketch an argument we learned from Tyler Lawson.

\begin{lemma}\label{lem:emendoperad}
If $R$ is a commutative ring, then the space of $\mathbb{E}_{\infty}$-structures on the Eilenberg--MacLane spectrum $HR$ that induce the given ring structure on $\pi_* HR = R$ is contractible.
\end{lemma}
\begin{proof}
This follows from the observation that the endomorphism operad of $HR$ in spectra is discrete; indeed, for all $m\ge 0$:
\begin{align*}
\Map((HR)^{\otimes m},HR) & \simeq \Map(\tau_{\le 0}(HR)^{\otimes m},HR) \\
& \simeq \Map(H(R^{\otimes m}),HR) \\
& \simeq \Hom(R^{\otimes m},R).
\end{align*}
In order to construct a coherent multiplication on $HR$, we have to pick out the union of path components corresponding to the iterated multiplication maps. This is a suboperad which is levelwise contractible, and hence we can replace it by an $\mathbb{E}_{\infty}$-operad. 
\end{proof}

To construct $H\Z$ as a Thom spectrum, we will combine the construction of $H\F_p$ as a Thom spectrum given in \Cref{hopkinsmahowaldthm} with the idea of intermediate Thom spectra from~\cite{beardsley_interthom}. Consider the fiber sequence 
\[
\xymatrix{S^3\langle 3 \rangle \ar[r] & S^3 \ar[r] & K(\Z,3)}
\]
which realizes the bottom of the Whitehead tower for $S^3$. Looping twice gives another fiber sequence
\begin{equation}\label{eq:s3fibseq}
\xymatrix{\Omega^2(S^3\langle 3 \rangle) \ar[r] & \Omega^2S^3 \ar[r]^-{\pi} & S^1\simeq \Omega^2K(\Z,3).}
\end{equation}
As in the previous section, let $f_p\colon \Omega^2S^3 \to BGL_1(S_p^0)$ be the free $\mathbb{E}_2$-map classifying the element $(1-p) \in \pi_0S_p^0$ and let $g_p$ be the composite 
\[
\xymatrix{g_p\colon \Omega^2(S^3\langle 3 \rangle) \ar[r] & \Omega^2S^3 \ar[r]^-{f_p} & BGL_1S_p^0.}
\]
For clarity, we will indicate the composite of a map $h\colon X \to BGL_1R$ with the inclusion $BGL_1R \to \Mod_R$ by the corresponding capital letter $H$; in particular, $Mh$ is the colimit of $H$.

\begin{thm}
There is an equivalence $Mg_p \simeq H\Z_p$ of $\mathbb{E}_2$-ring spectra.
\end{thm}
\begin{proof}
Following ideas of Beardsley~\cite{beardsley_interthom}, we will first construct an intermediate Thom spectrum $M\phi_p$ associated to a map $\phi_p\colon S^1 \to BGL_1(Mg_p)$ from the base of the fiber sequence \eqref{eq:s3fibseq}. The properties of this spectrum allow us then to determine the homotopy groups of $Mg_p$, from which the claim will follow.

Let $\Phi_p $ be the operadic Kan extension of $F_p$ along $\pi$ with respect to the $\mathbb{E}_2$-operad. By~\cite{beardsley_interthom}, $\Phi_p$ factors canonically through a map $\phi_p\colon S^1 \to BGL_1(Mg_p)$, so we obtain the following (non-commutative) diagram
\[
\xymatrix{\Omega^2(S^3\langle 3 \rangle) \ar[r] \ar@/^1pc/[rr]^-{g_p} & \Omega^2S^3 \ar[r]_-{f_p} \ar[d]_{\pi} & BGL_1S_p^0 \ar[r] & \Sp \\
& S^1 \ar@{-->}[r]_-{\phi_p} & BGL_1Mg_p. \ar[ru]}
\]
We will now identify the Thom spectrum of $\phi_p$ in two different ways. On the one hand, by~\cite{beardsley_interthom}, there is an equivalence $M\phi_p \simeq Mf_p$ of $\mathbb{E}_1$-ring spectra, which in turn is equivalent to $H\F_p$ by \Cref{hopkinsmahowaldthm}. On the other hand, \Cref{charpthomcomparison} shows that $M\phi_p$ is also the versal $\mathbb{E}_0$-$Mg_p$-algebra of characteristic $\chi_p$, where $\chi_p \in \pi_0(Mg_p)^{\times}$ is the element corresponding to $\phi_p$. Combining these two descriptions, we obtain a fiber sequence
\begin{equation}\label{eq:thgfibseq}
\xymatrix{Mg_p \ar[r]^-{1-\chi_p} & Mg_p \ar[r] & H\F_p}
\end{equation}
of spectra. By construction, $Mg_p$ is a connected $p$-complete $\mathbb{E}_2$-ring spectrum.

We now claim that $Mg_p$ is of finite type. To this end, first observe that it follows from \Cref{algmap-is-orient} that $Mg_p$ is $Mf_p\simeq H\F_p$-oriented. Therefore, we get that
\[
H_*(Mg_p;\F_p) \cong H_*(\Omega^2(S^3\langle 3 \rangle);\F_p).
\]
In particular, the homology of $Mg_p$ is finitely generated in each degree, so an argument with Serre classes shows that $\pi_*((Mg_p)_p) \cong \pi_*(Mg_p)$ is also finitely generated over $\Z_p$ in each degree. Furthermore, the Hurewicz theorem implies that
\[
\F_p \otimes \pi_0Mg_p \cong H_0(\Omega^2(S^3\langle 3\rangle);\F_p) \cong \F_p,
\]
because $\Omega^2(S^3\langle 3\rangle)$ is connected. In particular, $\pi_0Mg_p$ is a cyclic $\Z_p$-module as it is finitely generated over $\Z_p$.  

Next, consider the long exact sequence of homotopy groups associated to \eqref{eq:thgfibseq}, which degenerates to a short exact sequence
\[
\xymatrix{0 \ar[r] & \pi_0Mg_p \ar[r]^-{1-\chi_p} & \pi_0Mg_p \ar[r] & \F_p \ar[r] & 0}
\]
as well as isomorphisms $1-\chi_p\colon \pi_iMg_p \xrightarrow{\sim} \pi_iMg_p$ for $i \ge 1$. Using that $Mg_p$ is of finite type once more, one easily sees that this forces $1-\chi_p = p$ and
\[
\pi_*Mg_p \cong
	\begin{cases}
		\Z_p & i=0 \\
		0 & i \ne 1,
	\end{cases}
\]
hence $Mg_p \simeq H\Z_p$ as spectra. By virtue of \Cref{lem:emendoperad}, we conclude that this equivalence is also one of $\mathbb{E}_2$-ring spectra.
\end{proof}

As in \cite[Section~9.3]{blumbergthhthom}, we may glue these maps together to construct $H\Z$ as a Thom spectrum as well. Here, we work in the category of all spectra, not just $p$-complete ones. 

\begin{cor}
There is a map $g\colon \Omega^{2}S^3 \to BGL_1S^0$ whose associated Thom spectrum is equivalent to $H\Z$ as $\mathbb{E}_2$-ring spectra.
\end{cor}

\begin{rmk}
Recent work of Kitchloo~\cite{kitchloothom} gives a description of $H\Z/p^k$ as a Thom spectrum for all primes $p$ and all $k\ge 1$ with the exception of the case $(p,n) = (2,2)$. 
\end{rmk}

\subsection{Topological Hochschild homology and the cotangent complex}

\Cref{charpthomcomparison} allows to transport results proven for Thom spectra to versal algebras and vice versa. We illustrate this idea with two examples, computing the topological Hochschild homology and the cotangent complex of these algebras. These results generalize previous computations of Szymik~\cite{szymikcharacteristics1}. In this section, $R$ is an $\Es{\infty}$-ring spectrum. 

\begin{prop}\label{thhthom}
Let $\modmod[\Es{n}]{R}{\chi}$ be the versal $\Es{n}$-algebra of characteristic $\chi$ corresponding to a map $\bar{f}\colon \Omega^{n}\Sigma^{n}S^{k+1} \to BGL_1{R}$. If either $n\ge 3$ or $n\ge 2$ and the versal $\Es{n}$-algebra $\modmod[\Es{n}]{R}{\chi}$ of characteristic $\chi$ admits an $\Es{\infty}$-refinement, then there is an equivalence
\[\xymatrix{B\Omega^{n}\Sigma^{n}S^{k+1}_+ \otimes\modmod[\Es{n}]{R}{\chi}\ar[r]^-{\sim} & \THH(\modmod[\Es{n}]{R}{\chi}).}\]
\end{prop}
\begin{proof}
  Using \Cref{charpthomcomparison}, it is enough to compute the topological Hochschild homology of $M\bar{f}$. If $n\ge 3$, the claim thus follows from~\cite[Thm. 3]{bcsthhthom}, while the $n\ge 2$ case is covered by~\cite[1.6]{blumbergthhthom}.
\end{proof}

The deformation theory of an $\Es{n}$ $R$-algebra $A$ is captured by its $\Es{n}$-cotangent complex $\L{n}{A}$, considered as an object in the $\Es{n}$-monoidal category $\Mod_A^{\Es{n}} = \Sp((\Alg_R^{\Es{n}})_{/A})$ of $\Es{n}$ $A$-modules, see~\cite{franciscotangent} or~\cite{HA}.

\begin{prop}\label{ccthom}
If $n\ge 1$ and $\modmod[\Es{n}]{R}{\chi}$ a the versal $\Es{n}$-algebra of characteristic $\chi\colon \Sigma^kR \to R$, then there is an equivalence 
\[\xymatrix{\Sigma^{k+1} \free{1}(\Sigma^{n-1} R) \otimes \modmod[\Es{n}]{R}{\chi} \ar[r]^-{\sim} &  \L{n}{\modmod[\Es{n}]{R}{\chi}}}\]
of $\Es{n}$ $\modmod[\Es{n}]{R}{\chi}$-modules. 
\end{prop}

\begin{proof}
Firstly, the cotangent complex of the free $\Es{n}$-algebra $\free{n}\Sigma^kR$ can be computed as an $\Es{n}$ $\free{n}\Sigma^kR$-module using \cite[2.12, 2.25]{franciscotangent} as follows:
\[
\L{n}{\free{n}\Sigma^kR} \simeq \Sigma^k U_{\free{n}\Sigma^kR} \simeq \Sigma^k\free{1}(\Sigma^{k+n-1} R) \otimes {\free{n}(\Sigma^kR)},\]
where $U_{\free{n}\Sigma^kR}$ denotes the enveloping algebra of ${\free{n}\Sigma^kR}$ as in \cite{franciscotangent}. Secondly, the natural cofiber sequence for computing the relative cotangent complex~\cite[2.11]{franciscotangent} specializes to give a cofiber sequence
\[\xymatrix{\L{n}{\free{n}\Sigma^kR} \otimes_{\free{n}\Sigma^kR} R \ar[r] & \L{n}{R} \ar[r] & \L{n}{R/\free{n}\Sigma^kR}.}\]
Since $\L{n}{R}$ is contractible, we get 
\begin{align*}
\L{n}{R/\free{n}\Sigma^kR} & \simeq \Sigma \L{n}{\free{n}\Sigma^kR} \otimes_{\free{n}\Sigma^kR} R \\
& \simeq \Sigma^{k+1} \free{1}(\Sigma^{k+n-1} R) \otimes \free{n}\Sigma^kR \otimes_{\free{n}\Sigma^kR} R \\
& \simeq \Sigma^{k+1} \free{1}(\Sigma^{k+n-1} R).
\end{align*}
The base-change formula~\cite[7.3.3.7]{HA} applied to the defining pushout diagram
\[\xymatrix{\free{n} \Sigma^kR \ar[r]^-{\bar{0}} \ar[d]_{\bar{\chi}} & R \ar[d]^{\psi} \\
R \ar[r] & \modmod{R}{\chi}}\]
now gives the desired equivalence
\[\xymatrix{\L{n}{\modmod{R}{\chi}} \ar[r]^-{\sim} & \psi_!\L{n}{R/\free{n}\Sigma^kR} \simeq \Sigma^{k+1} \free{1}(\Sigma^{k+n-1} R) \otimes \modmod{R}{\chi}}\]
of $\Es{n}$ $\modmod{R}{\chi}$-modules. 
\end{proof}

By \Cref{charpthomcomparison}, we can translate this result immediately into a statement about the $\Es{n}$-cotangent complex of certain Thom spectra.

\begin{cor}
Let $n\ge 1$ and $f\colon S^{k+1} \to BGL_1(R)$ with corresponding $n$-fold loop map $\bar{f}\colon \Omega^n\Sigma^nS^{k+1} \to BGL_1(R)$, then there is an equivalence
\[\xymatrix{\Sigma^{k+1} \free{1}(S^{k+n-1}) \otimes M\bar{f} \ar[r]^-{\sim} &  \L{n}{M\bar{f}}}\]
of $\Es{n}$ $M\bar{f}$-modules. 
\end{cor}

\bibliographystyle{alpha}
\bibliography{bibliography}

\end{document}